\documentclass[11pt]{article}%
\usepackage{amssymb}
\usepackage{amsfonts}
\usepackage{amsmath}
\usepackage{euler}
\usepackage{graphicx}%
\providecommand{\U}[1]{\protect\rule{.1in}{.1in}}
\newtheorem{theorem}{Theorem}

\newtheorem{corollary}[theorem]{Corollary}

\newtheorem{lemma}[theorem]{Lemma}

\newtheorem{proposition}[theorem]{Proposition}

\newenvironment{proof}[1][Proof]{\noindent\textit{#1.} }{}
\begin{document}

\title{Metric Double Complements of Convex Sets}
\author{Douglas S. Bridges}
\maketitle

\begin{abstract}%
\noindent
In constructive mathematics the metric complement of a subset $S$ of a metric
space $X$ is the set $-S$ of points in $X$ that are bounded away from $S$. In
this note we discuss, within Bishop's constructive mathematics, the connection
between the metric double complement, $-(-K)$, and the logical double
complement, $\lnot\lnot K$, where $K$ is a convex subset of a normed linear
space $X$. In particular, we prove that if $K$ has inhabited interior, then
$-(-K)=\left(  \lnot\lnot K\right)  ^{\circ}$, that the hypothesis of
inhabited interior can be dropped\ in the finite-dimensional case, and that we
cannot constructively replace $(\lnot\lnot K)^{\circ}$ by $K^{\circ}$ in these results.

\end{abstract}%

\normalfont\sf

\section{Preliminaries}%

\noindent
In this article we discuss, within Bishop's constructive mathematical
framework (\textbf{BISH)},\footnote{%
\normalfont\sf
Very roughly, \textbf{BISH} is mathematics carried out using intuitionistic
logic and an appropriate set- or type-theoretic framework \cite{Handbook}.}
double complements of convex sets in a normed linear space. We assume that the
reader has access to, or some familiarity with, the constructive theory of
metric and normed linear spaces, as can be found in
\cite{Bishop,Handbook,BVtech}. However, it may help to put on record a few
elementary notions in constructive analysis.

If $\left(  X,\rho\right)  $ is a metric space,\footnote{%
\normalfont\sf
We use $\rho$ to denote the metric on any metric space.} a point $x\ $is
\textbf{bounded away} from a subset $S$ if there exists $r>0$ such that
$\rho(x,y)\geq r$ for all $y\in S$. The set of points that are bounded away
from $S$ is an open set called the \textbf{metric complement }of $S$ in $X$
and is denoted by $-S$. The\textbf{\ metric double complement} of $S$ is
$-(-S)$. If $S$ is \textbf{located} in $X$----that is,%
\[
\rho(x,S)\equiv\inf\left\{  \left\Vert x-y\right\Vert :y\in S\right\}
\]
exists for each $x\in X\,$---then $-S=\left\{  x\in X:\rho(x,S)>0\right\}  $.
The \textbf{logical complement }of $S$ in $X$ is\textbf{\ }%
\[
\lnot S\equiv\left\{  x\in X:\forall y\in S(x\notin y)\right\}
\]
and the \textbf{logical double complement} of $S$ is $\lnot(\lnot S)$, also
written $\lnot\lnot S$. There is yet another complement to consider: the
\textbf{strong complement }of $S$ is%
\[%
\mathord{\sim}%
S\equiv\left\{  x\in X:\forall y\in S(x\neq y)\right\}  ,
\]
where $x\neq y$ means $\rho(x,y)>0$ (a condition constructively stronger than
$\lnot(x=y))$; the \textbf{strong double complement} of $S$ is, as expected, $%
\mathord{\sim}%
(%
\mathord{\sim}%
S)$. Note that $-S=-\overline{S}\subset%
\mathord{\sim}%
S\ $and that $S^{\circ}\subset-(-S)$, where $\overline{S}$ and $S^{\circ}$
denote, respectively, the closure and interior of $S$.

We have little to say here about strong complements, our main concern being
the connection between$\ -(-S)$, $\left(  \lnot\lnot S\right)  ^{\circ}$, and
$S^{\circ}$ in the case where $S$ is convex.\footnote{%
\normalfont\sf
Of course classically (that is, with classical logic) $\lnot\lnot S=S$ for any
$S$. But Proposition \ref{oct03p2} below shows that this identity cannot, in
general, be proved constructively even when $S$ is a convex subset of
$\mathbf{R}$.} However, before looking at convex sets, we gather here a number
of elementary results about complements in metric spaces.

\begin{lemma}
\label{oct14l1}Let $S$ and $T$ be subsets of a metric space $X$ such that
$S\subset\lnot T$ and $S$ is open. Then $S\subset-T$.
\end{lemma}

\begin{proof}
Given $x\in S$, choose $r>0$ such that $B(x,r)\subset S$. Then $B(x,r)\subset
\lnot T$, so $\rho(x,y)\geq r$ for each $y\in T$ and therefore $x\in-T$.%
\hfill
$\square$
\end{proof}

\begin{lemma}
\label{oct03l2}Let $S$ be a subset of a metric space $X$. Then $-(-S)=\left(
\lnot(-S)\right)  ^{\circ}.$
\end{lemma}

\begin{proof}
Since $-(-S)\subset\lnot(-S)$ and metric complements are open, $-(-S)\subset
\left(  \lnot(-S)\right)  ^{\circ}$. On the other hand, $\left(
\lnot(-S)\right)  ^{\circ}\subset\lnot(-S)$, so by Lemma \ref{oct14l1},
$\left(  \lnot(-S)\right)  ^{\circ}\subset-(-S).$%
\hfill
$\square$
\end{proof}

\begin{lemma}
\label{mar21l0}Let $X$ be a metric space, and $S$ a subset of $X$ such that
$-(-S)\subset\lnot\lnot S$. Then $-(-S)=\left(  \lnot\lnot S\right)  ^{\circ}$.
\end{lemma}

\begin{proof}
Since $-(-S)$ is open, $-(-S)\subset\left(  \lnot\lnot S\right)  ^{\circ}$. On
the other hand, $\left(  \lnot\lnot S\right)  ^{\circ}\subset\lnot\lnot
S\subset\lnot(-S)$, so by Lemma \ref{oct14l1}, $\left(  \lnot\lnot S\right)
^{\circ}\subset-(-S)$.%
\hfill
$\square$
\end{proof}

\begin{lemma}
\label{mar21l2}If $S$ is a located subset of a metric space $X$, then
$\lnot(-S)\subset\overline{S}$.
\end{lemma}

\begin{proof}
Since $S$ is located, for each $x\in\lnot(-S)$ we have $\lnot\left(
\rho(x,S)>0\right)  $, so $\rho(x,S)=0$ and therefore $x\in\overline{S}$.%
\hfill
$\square$
\end{proof}

\begin{lemma}
\label{sept27l2}Let $S$ be a closed located subset of a metric space $X$. Then
$-(-S)=-(%
\mathord{\sim}%
S)=-(\lnot S)=S^{\circ}$.
\end{lemma}

\begin{proof}
Since $S^{\circ}\subset\lnot\lnot S$, Lemma \ref{oct14l1} yields $S^{\circ
}\subset-(\lnot S)$. From this, Lemma \ref{oct03l2}, and Lemma \ref{mar21l2}
we obtain%
\[
-(-S)=\left(  \lnot(-S)\right)  ^{\circ}\subset\overline{S}^{\circ}=S^{\circ
}\subset-(\lnot S)\text{.}%
\]
and therefore $-(-S)=-(\lnot S)=S^{\circ}$. Moreover, since $-S\subset%
\mathord{\sim}%
S\subset\lnot S$, we have%
\[
-\left(  \lnot S\right)  \subset-(%
\mathord{\sim}%
S)\subset-\left(  -S\right)  =-(\lnot S)\
\]
and therefore $-(-S)=-(%
\mathord{\sim}%
S)=-(\lnot S)$.%
\hfill
$\square$
\end{proof}

\begin{lemma}
\label{oct15l1}Let $X$ be a metric space, $S$ a closed located subset of $X$,
and $K$ a subset of $S$ such that $S\cap\lnot K=\varnothing$. Then $S^{\circ
}=-(-K)$.
\end{lemma}

\begin{proof}
First we have $S^{\circ}\subset S\subset\lnot(-K)$, so by Lemma \ref{oct14l1},
$S^{\circ}\subset-(-K)$. On the other hand, $-(-K)\subset-(-S)=S^{\circ}$, the
equality following from Lemma \ref{sept27l2},
\hfill
$\square$
\end{proof}

\section{Convex sets}

In this section we focus on double complements in convex subsets of a normed
linear space. But first we derive some important properties of interiors of
such sets.\footnote{%
\normalfont\sf
Several of our proofs may be well known, but we choose to include them for
completeness of exposition.}

\begin{lemma}
\label{oct03l1}Let $x,y$ be points in a normed space $X$, let $z=\left(
1-\lambda\right)  x+\lambda y$, where $0\leq\lambda\leq1$, and let $r>0$. If
$\zeta\in B(z,r)$, then $\zeta=(1-\lambda)\xi+\lambda\eta$, where $\xi
=x+\zeta-z\in B(x,r)$ and $\eta=y+\zeta-z\in\eta\in B(y,r)$.
\end{lemma}

\begin{proof}
Let $\zeta,\xi,$ and $\eta$ be as in the hypotheses. Then $\xi\in B(x,r)$,
$\eta\in B(y,r)$, and $\left(  1-\lambda\right)  \xi+\lambda\eta=\left(
1-\lambda\right)  x+\lambda y+\zeta-z=\zeta$.%
\hfill
$\square$
\end{proof}

%

\medskip
The next lemma is a rewording of \cite[5.1.1]{BVtech}.

\begin{lemma}
\label{oct06l1}Let $K$ be a convex subset of a normed space $X$, let $x\in
K^{\circ}$, and let $r>0$ be such that $B(x,r)\subset K$. Let $y\neq x$,
$0<\lambda<1$, and $z=(1-\lambda)x+\lambda y$. If $B(y,(1-\lambda)r)$
intersects $K$, then $B(z,\left(  1-\lambda\right)  ^{2}r)\subset K$.
\end{lemma}

\begin{proposition}
\label{oct03p1}Let $K$ be a convex subset of a normed space $X$. Then
$K^{\circ}$ is convex.
\end{proposition}

\begin{proof}
Let $x,y\in K^{\circ}$. There exists $r>0$ such that $B(x,r)\subset K^{\circ}$
and $B(y,r)\subset K^{\circ}$. Let $0\leq\lambda\leq1$ and $z=(1-\lambda
)x+\lambda y$. By Lemma \ref{oct03l1}, for each $\zeta\in B(z,r)$ there exist
$\xi\in B(x,r)\subset K$ and $\eta\in B(y,r)\subset K$ such that
$\zeta=\left(  1-\lambda\right)  \xi+\lambda\eta$, which belongs to the convex
set $K$. Thus $B(z,r)\subset K$ and therefore $z\in K^{\circ}$.%
\hfill
$\square$
\end{proof}

\begin{proposition}
\label{oct06p1}Let $K$ be a convex set with inhabited interior in a normed
linear space $X$. Then $K^{\circ}$ is dense in $\overline{K}$.
\end{proposition}

\begin{proof}
It will suffice to prove that $K^{\circ}$ is dense in $K$. Fix $x_{0}\in
K^{\circ}$ and $r>0$ such that $B(x_{0},r)\subset K^{\circ}$, and consider any
$y\in K$. Either $\left\Vert x_{0}-y\right\Vert <r$ and therefore $y\in
K^{\circ}$, or else, as we may suppose, $x_{0}\neq y$. Given $\varepsilon$
with $0<\varepsilon<2\left\Vert x_{0}-y\right\Vert $, we see that%
\[
0<\lambda\equiv1-\frac{\varepsilon}{2\left\Vert x_{0}-y\right\Vert }<1\text{.}%
\]
Let $z=(1-\lambda)x_{0}+\lambda y$, $\zeta\in B(z,\left(  1-\lambda\right)
r)$, and%
\[
\xi=\frac{1}{1-\lambda}\zeta-\frac{\lambda}{1-\lambda}y.
\]
Then%
\[
x_{0}=\frac{1}{1-\lambda}z-\frac{\lambda}{1-\lambda}y,
\]
so%
\[
\left\Vert \xi-x_{0}\right\Vert =\frac{1}{1-\lambda}\left\Vert \zeta
-z\right\Vert <r
\]
and therefore $\xi\in K$. Since $\zeta=(1-\lambda)\xi+\lambda y$, it follows
that $\zeta\in K$. Hence $B(z,(1-\lambda)r)\subset K$ and so $z\in K^{\circ}$.
On the other hand, $\left\Vert y-z\right\Vert =\left(  1-\lambda\right)
\left\Vert x_{0}-y\right\Vert =\varepsilon/2$. Thus for each $y\in K$ and each
$\varepsilon>0$, there exists $z\in K^{\circ}$ such that $\left\Vert
y-z\right\Vert <\varepsilon$.%
\hfill
$\square$
\end{proof}

\begin{corollary}
\label{oct06c1}Let $K$ be a convex set in a normed space $X$. If $x\in
K^{\circ}$ and $y\in\overline{K}$, then $[x,y)\subset K^{\circ}$.
\end{corollary}

\begin{proof}
Choose $r>0$ such that $B(x,r)\subset K^{\circ}$. Let $0\leq\lambda<1$ and
$z=\left(  1-\lambda\right)  x+\lambda y$. Either $\left\Vert z-x\right\Vert
<r$ and $z\in K^{\circ}$, or else $z\neq x$. In the latter case, $0<\lambda<1$
and%
\[
\left\Vert y-x\right\Vert =\left\Vert \left(  1-\frac{1}{\lambda}\right)
x+\frac{1}{\lambda}z-x\right\Vert =\frac{1}{\lambda}\left\Vert z-x\right\Vert
>0,
\]
so $y\neq x$. Since $y\in\overline{K}$, Proposition \ref{oct06p1} shows that
$B(y,\left(  1-\lambda\right)  r)$ intersects $K^{\circ}$; whence, by Lemma
\ref{oct06l1}, $B(z,\left(  1-\lambda\right)  ^{2}r)\subset K$ and therefore
$z\in K^{\circ}$.%
\hfill
$\square$
\end{proof}

\begin{corollary}
\label{oct07c2}Let $K$ be a convex set with inhabited interior in a normed
linear space $X$. Then $K^{\circ}=\overline{K}^{\circ}$.
\end{corollary}

\begin{proof}
It will suffice to prove that $\overline{K}^{\circ}\subset K^{\circ}$. Fix
$x_{0}\in K$ and $r>0$ such that $B(x_{0},r)\subset K^{\circ}$. Given any
$y\in\overline{K}^{\circ}$, choose $s>0$ such that $\overline{B}%
(y,s)\subset\overline{K}$. Either $\left\Vert y-x_{0}\right\Vert <r$ and
therefore $y\in K^{\circ}$, or else $y\neq x_{0}$. in the latter event, let%
\[
\alpha\equiv1+\frac{s}{\left\Vert y-x_{0}\right\Vert }.
\]
and $z=(1-\alpha)x_{0}+\alpha y$. Then $\left\Vert z-y\right\Vert =\left(
\alpha-1\right)  \left\Vert y-x_{0}\right\Vert =s$, so $z\in\overline
{B}(y,s)\ $and therefore $z\in\overline{K}$. Hence, by Corollary
\ref{oct06c1}, $[x_{0},z)\subset K^{\circ}$. But $y=\left(  1-\frac{1}{\alpha
}\right)  x_{0}+\frac{1}{\alpha}z$, where $0<1/\alpha<1$; so $y\in\lbrack
x_{0},z)$ and therefore $y\in K^{\circ}$. It follows that $\overline{K}%
^{\circ}\subset K^{\circ}$.%
\hfill
$\square$
\end{proof}

%

\medskip
Our next lemma is a weak substitute for the classical logic rule
$\lnot(p\wedge q)=\lnot p\vee\lnot q$, which is not a theorem of
intuitionistic logic.

\begin{lemma}
\label{sept30l1}$(\lnot\lnot p\wedge\lnot\lnot q\Rightarrow\lnot\lnot(p\wedge
q))$
\end{lemma}

\begin{proof}
Since%
\[
\lnot\lnot p\wedge\lnot\lnot q\Rightarrow(p\wedge\lnot(p\wedge q)\Rightarrow
\lnot q\Rightarrow\lnot q\wedge\lnot\lnot q)
\]
we see that%
\begin{equation}
\lnot\lnot p\wedge\lnot\lnot q\Rightarrow(\lnot(p\wedge q)\Rightarrow\lnot
p\Rightarrow\lnot p\wedge\lnot\lnot p)\Rightarrow\lnot\lnot(p\wedge
q).\tag*{$\square$}%
\end{equation}

\end{proof}

At last we have our first result about double complements of convex sets.

\begin{proposition}
\label{sept30p1}Let $K$ be a convex subset of a normed space $X$. Then
$\lnot\lnot K$ is convex.
\end{proposition}

\begin{proof}
Given $x,y\in\lnot\lnot K$ and $\lambda\in\left[  0,1\right]  $, suppose that
$z\equiv\left(  1-\lambda\right)  x+\lambda y$ belongs to $\lnot K$. Then,
since $K$ is convex, $\lnot(x\in K\wedge y\in K)$. But, by Lemma
\ref{sept30l1}, $\lnot\lnot(x\in K\wedge y\in K)$. This contradiction ensures
that $z\notin\lnot K$, so $z\in\lnot\lnot K$.%
\hfill
$\square$
\end{proof}

\begin{proposition}
\label{oct01p1a}If $K$ is a convex subset of a normed space $X$ such that
$-(-K)\subset\lnot\lnot K$, then $-(-K)$ is convex.
\end{proposition}

\begin{proof}
Let $x,y\in-(-K)$. Since metric complements are open, there exists $r>0$ such
that $B(x,r)\subset-(-K)$ and $B(y,r)\subset-(-K)$. Let $0\leq\lambda\leq1$
and $z=(1-\lambda)x+\lambda y$. By Lemma \ref{oct03l1}, for each $\zeta\in
B(z,r)$ there exist $\xi\in B(x,r)$ and $\eta\in B(y,r)$ such that
$\zeta=\left(  1-\lambda\right)  \xi+\lambda\eta$. Then $\xi,\eta
\in-(-K)\subset\lnot\lnot K$, so by Proposition \ref{sept30p1}, $\zeta\in
\lnot\lnot K$ and hence $\zeta\in\lnot(-K)$. It follows that $B(z,r)\subset
\lnot(-K)$; whence $z\in-(-K)$.%
\hfill
$\square$
\end{proof}

%

\medskip
Geometric intuition suggests that if $X$ is a normed linear space and $S$ is a
convex subset of $X$, then $-(-S)\subset S^{\circ}$. This intuition (trivially
true if $S^{\circ}$ is empty, and almost trivially true of $S$ is a ball in a
normed linear space) is confirmed classically by applying the law of excluded
middle to the conclusion of our first theorem.

\begin{theorem}
\label{mar20t1}Let $X$ be a normed linear space, and $K$ a convex subset of
$X$ with inhabited interior. Then $-(-K)=\left(  \lnot\lnot K\right)  ^{\circ}
$.
\end{theorem}

\begin{proof}
In view of Lemma \ref{mar21l0}, it will suffice to prove that $-(-K)\subset
\lnot\lnot K$. Fix $x_{0}\in K^{\circ\text{ }}$and $t>0$ such that
$B(x_{0},t)\subset K^{\circ}$. Given $a\in-(-K)$, we have either $\left\Vert
a-x_{0}\right\Vert <t$ and therefore $a\in K^{\circ}\subset\left(  \lnot\lnot
K\right)  ^{\circ}$, or, as we may suppose, $a\neq x_{0}$. Consider first the
case where $a=0$, when there exists $r>0$ such that $B(0,r)\subset-(-K)$.
Heading towards a contradiction, assume that $0\in\lnot K$. Choose $R>r$ such
that $B(-x_{0},t)\subset B(0,R)$, let $\lambda=r/R$, and let $x_{1}=-\lambda
x_{0}$. Then $B(x_{1},\lambda t)=\lambda B(-x_{0},t)$. Also, for each $z\in
B(-x_{0},t)$, $\left\Vert z\right\Vert <R$ and therefore $\left\Vert \lambda
z\right\Vert =\frac{r}{R}\left\Vert z\right\Vert <r$; whence%
\begin{equation}
B(x_{1},\lambda t)\subset B(0,r)\subset-(-K). \label{A1}%
\end{equation}
On the other hand, if $z\in K\cap B(x_{1},\lambda t)$, then $-\frac{1}%
{\lambda}z\in$ $B(x_{0},t)\subset K$ and therefore%
\[
0=\frac{\lambda}{1+\lambda}\left(  -\frac{1}{\lambda}z\right)  +\frac
{1}{1+\lambda}z\in K,
\]
a contradiction. Hence $z\in\lnot K$ for each $z\in B(x_{1},\lambda t)$, from
which it follows that $x_{1}\in-K$, contradicting (\ref{A1}). We conclude that
if $0\in-(-K)$, then $0\in\lnot\lnot K$. This disposes of the case $a=0$. In
the general case let $L\equiv\left\{  x-a:x\in K\right\}  $, which is a convex
subset of $X$ with inhabited interior such that $0\in-(-L)$. If $a\in\lnot K$,
then $0\in\lnot L$, which contradicts the first part of the proof, so
$a\in\lnot\lnot K$. Thus $-(-K)\subset\lnot\lnot K$.%
\hfill
$\square$
\end{proof}

\begin{corollary}
\label{oct06c2}If $K$ is a convex subset with inhabited interior in a normed
space $X$, then $-(-K)$ is convex.
\end{corollary}

\begin{proof}
By Theorem \ref{mar20t1}, $-(-K)\subset\lnot\lnot K$, so we can apply
Proposition \ref{oct01p1a}.%
\hfill
$\square$
\end{proof}

\begin{corollary}
\label{oct04c1}Let $K$ be a convex set with inhabited interior in a normed
linear space $X$. Then $-(-K)$ is dense in $\lnot\lnot K$.
\end{corollary}

\begin{proof}
By Theorem \ref{mar20t1}, $-(-K)=\left(  \lnot\lnot K\right)  ^{\circ}$. Since
$\lnot\lnot K\ $is convex by Proposition \ref{sept30p1}, and $\left(
\lnot\lnot K\right)  ^{\circ}\supset K^{\circ}$, we see from Proposition
\ref{oct06p1} that $-(-K)\ $is dense in $\overline{\lnot\lnot K}$ and hence in
$\lnot\lnot K$.%
\hfill
$\square$
\end{proof}

\section{The finite-dimensional case}

For finite-dimensional spaces we can drop from Theorem \ref{mar20t1} the
hypothesis that $K$ have inhabited interior. To do this, we bring into play
the constructive theory of simplices, of which mention two things in
particular.\footnote{%
\normalfont\sf
For more information about the constructive theory of simplices we refer the
reader to \cite{DSBsimplex}.}

\begin{proposition}
\label{sept15p1}Let $\Sigma$ be an $n$-simplex \ with vertices $a_{1}%
,\ldots,a_{n+1}$ in an $n$-dimen- sional normed linear space $X$. Then the
\textbf{barycentre}, $\frac{1}{n+1}\sum_{k=1}^{n+1}a_{k}$, of $\Sigma$ belongs
to $\Sigma^{\circ}$.
\end{proposition}

\begin{proof}
This follows from Proposition 21 and Corollary 23 of \cite{DSBsimplex}.%
\hfill
$\square$
\end{proof}

\begin{proposition}
\label{sept16c1}Let $\Sigma$ be an $n$-simplex with vertices $a_{1}%
,\ldots,a_{n+1}$ in an $n$-dimen- sional Banach space, and let $c\ $be an
interior point of $\Sigma$. There exists $\delta>0$ such that if $x_{k}\in X$
and $\left\Vert x_{k}-a_{k}\right\Vert <\delta$ for each $k\leq n+1$, then the
convex hull of the points $x_{1},\ldots,x_{n+1}$ is an $n$-simplex with $c$ in
its interior.
\end{proposition}

\begin{proof}
This follows from Theorem 29 and Proposition 21 of \cite{DSBsimplex}.%
\hfill
$\square$
\end{proof}

\begin{theorem}
\label{mar20t2}If $K$ is an inhabited convex subset of a finite-dimensional
normed linear space $X$, then $-(-K)=$ $(\lnot\lnot K)^{\circ}$.
\end{theorem}

\begin{proof}
We may assume that $0\in K$ (see the last part of the proof of Theorem
\ref{mar20t1}). Let $n=\dim X$. Given $x\in-(-K)$, suppose that $x\in\lnot K$.
Suppose also that $K$ contains $n$ linearly independent elements $x_{1}%
,\ldots,x_{n}$ of $X$, and let $\Sigma$ be the $n$-simplex $\Sigma$ with
vertices $0,x_{1},\ldots,x_{n}$. Then $\Sigma$ is convex and each of its
vertices belongs to the convex set $K$, so $\Sigma\subset K$. By Proposition
\ref{sept15p1}, the interior of $\Sigma$ is inhabited, as therefore is the
interior of $K$. Hence, by Theorem \ref{mar20t1}, $x\in\lnot\lnot K$, a
contradiction. Thus, in fact, $K$ cannot contain $n$ linearly independent
vectors in $X$. Now suppose that for some integer $k$ with $1<k\leq n$, we
have proved that $K$ cannot contain $k$ linearly independent vectors. Assume
that $K$ contains $k-1$ linearly independent vectors spanning a $(k-1)$%
-dimensional linear space $V$. If $K\cap-V$ is inhabited, then $K$ contains
$k$ linearly independent vectors, a contradiction. Since $V$ is
finite-dimensional and hence located and complete, it follows that
$K\subset\overline{V}=V$. On the other hand, if $x\in-V$, then $x\in-K$, a
contradiction; so $x\in\overline{V}=V$. Thus $K$ is an inhabited convex subset
of $V$ containing $k-1$ linearly independent vectors, $x$ $\in V\cap-(-K)$,
and $x\in\lnot K$. The first part of this proof, with $k$ replacing $n$, shows
that this is impossible. Hence $K$ cannot contain $k-1$ linearly independent
vectors. Pursuing this downward inductive argument, we now see that $K=\{0\}$.
Then $x\in-(-K)=\varnothing$, a final contradiction from which we conclude
that $x\in\lnot\lnot K$. Hence $-(-K)\subset\lnot\lnot K\ $and it remains to
apply Lemma \ref{mar21l0}.%
\hfill
$\square$
\end{proof}

%

\medskip
When our convex set is located, we can strengthen Theorems \ref{mar20t1} and
\ref{mar20t2}.

\begin{theorem}
\label{mar24t1}Let $K$ be a located convex subset of a normed linear space $X
$. If either $K^{\circ}$ is inhabited or $X$ is finite-dimensional, then
$-(-K)=$ $K^{\circ}$.
\end{theorem}

\begin{proof}
By whichever of Theorem \ref{mar20t1} or Theorem \ref{mar20t2} is applicable
and Lemma \ref{oct03l2},%
\[
-(-K)=(\lnot\lnot K)^{\circ}\subset\left(  \lnot(-K)\right)  ^{\circ}=-(-K),
\]
from which the result follows.%
\hfill
$\square$
\end{proof}

%

\medskip
Let $\left\{  e_{1},\ldots,e_{n}\right\}  $ be the standard basis\footnote{%
\normalfont\sf
For each $k$, $e_{k}$ has $k^{\mathsf{th}}$ component $1$ and all other
components $0$.} of unit vectors in $\mathbf{R}^{n}$, and $u$ the vector in
$\mathbf{R}^{n}$ with each component $1$. Let $\Sigma^{n}$ be the simplex in
$\mathbf{R}^{n}$ with vertices $v_{1},\ldots,v_{n+1}$, where%
\[
v_{n}=\sqrt{1+\frac{1}{n}}e_{i}-n^{-3/2}\left(  \sqrt{n+1}+1\right)
u\ \ \ \ (1\leq i\leq n).
\]

and $v_{n+1}=\frac{1}{\sqrt{n}}u$. According to Wikipedia at

\begin{quote}
{\small https://en.wikipedia.org/wiki/Simplex\#Cartesian\_coordinates\_}

{\small for\_a\_regular\_n-dimensional\_simplex\_in\_Rn}
\end{quote}

%

\noindent
$\Sigma^{n}$ is a regular\footnote{%
\normalfont\sf
We don't actually need the fact that this simplex is regular. A rough
description of regularity is that the simplex is as symmetrical as you can
imagine. If you want to know precisely what a regular simplex is, see under
`simplex' in Wikipedia.} $n$-simplex, with barycentre $0\in\mathbf{R}^{n+1}$
and side length $\sqrt{2(n+1)/n}$, that is inscribed in the closed unit
Euclidean ball of $\mathbf{R}^{n}$. It follows that if $a\in\mathbf{R}^{n}$
and $r>0$, then%
\[
a+r\Sigma^{n}\equiv\left\{  a+rv:v\in\mathbf{\Sigma}^{n}\right\}
\]
is a regular $n$-simplex with barycentre $a$ that is inscribed in the
Euclidean ball $\overline{B}(a,r)$.

Using these facts, we have an interesting alternative proof for the
finite-dimensional case of Corollary \ref{oct07c2}.

\begin{proposition}
\label{mar24p1}If $K$ is a convex subset of a finite-dimensional normed linear
space $X$, then $\overline{K}^{\circ}=K^{\circ}$\textbf{.}
\end{proposition}

\begin{proof}
Again it suffices to prove that $\overline{K}^{\circ}\subset K^{\circ}$ We may
assume that $X=\mathbf{R}^{n}$ with the Euclidean norm. Given $\xi\in
\overline{K}^{\circ}$, choose $r>0$ such that $B(\xi,3r)\subset\overline
{K}^{\circ}$, and consider any $x\in B(\xi,r)$. Referring to the remarks
immediately preceding this proposition, let $a_{1},\ldots,a_{n+1}$ be the
vertices of a regular $n$-simplex $\Sigma$ inscribed in $\overline{B}(x,r)$
and having barycentre $x$. By Proposition \ref{sept16c1}, there exists
$\delta\in\left(  0,r\right)  $ such that if $x_{1},\ldots,x_{n+1}$ are points
of $X$ such that $\left\Vert a_{k}-x_{k}\right\Vert <\delta$ for each $k$,
then $x$ is a convex linear combination of the points $x_{k}$; in that case%
\[
\left\Vert x_{k}-\xi\right\Vert \leq\left\Vert x_{k}-a_{k}\right\Vert
+\left\Vert a_{k}-x\right\Vert +\left\Vert x-\xi\right\Vert <\delta+r+r=3r
\]
for each $k$, so $x_{k}\in\overline{K}$. Since $K$ is dense in $\overline{K}$,
we can ensure that each $x_{k}$ is in $K$. Then $x$ belongs to the convex set
$K$. It follows that $B(\xi,r)\subset K$ and therefore that $\xi\in K^{\circ}%
$. Thus $\overline{K}^{\circ}\subset K^{\circ}$.%
\hfill
$\square$
\end{proof}

\section{Can we do better?}

Finally, we show that Theorems \ref{mar20t1} and \ref{mar20t2} are the best we
can hope for in the constructive setting.

\begin{proposition}
\label{oct03p2}The following are equivalent.

\begin{enumerate}
\item[\emph{(i)}] If $K$ is a convex subset of $\mathbf{R}$ with inhabited
interior, such that $-(-K)$ is located, then $-(-K)=K^{\circ}$.

\item[\emph{(ii)}] The law of excluded middle.

\item[\emph{(iii)}] If $K$ is a convex set with inhabited interior in a normed
space $X$, then $-(-K)=K^{\circ}$.
\end{enumerate}
\end{proposition}

\begin{proof}
Let $p$ be a proposition such that $\lnot\lnot p$, and define%
\[
K\equiv\left[  0,1\right]  \cup\left\{  x\in\left[  0,2\right]  :p\right\}
\text{.}%
\]
Then $\left[  0,1\right]  \subset K\subset\left[  0,2\right]  $, so $K^{\circ
}$ is inhabited, and $K=\left[  0,2\right]  $ if and only if $p$. If $x\ $and
$y$ belong to $K$, then either both $x$ and $y$ are in $\left[  0,1\right]  $,
or else one of them belongs to $\left\{  x\in\left[  0,2\right]  :p\right\}  $
and $K=\left[  0,2\right]  $; in either case, $\left[  x,y\right]  \subset K$,
so $K$ is convex. If $\left[  0,2\right]  \cap-K$ is inhabited, then $\lnot
p$, which is absurd; so $\left[  0,2\right]  \cap-K=\varnothing$. Since
$\left[  0,2\right]  $ is closed and located in $\mathbf{R}$, Lemma
\ref{oct15l1} shows that $\left(  0,2\right)  =\left[  0,2\right]  ^{\circ
}=-(-K)$, which is therefore located. Moreover, $3/2=-(-K)$, so if (i) holds,
then $3/2\ $must belong to $\left\{  x\in\mathbf{R}:p\right\}  $ and therefore
$p$ holds. It follows that if (i) holds, then $(\lnot\lnot p\Rightarrow p)$
for every proposition $p$; whence (i) implies (ii). Referring to Theorem
\ref{mar20t1}, we see that (ii) implies (iii). It is trivial that (iii)
implies (i).%
\hfill
$\square$
\end{proof}

%

\medskip

Proposition \ref{oct03p2} shows that, working constructively, we cannot remove
`$\lnot\lnot$' from the conclusion of either Theorems \ref{mar20t1} and
\ref{mar20t2}, even under the additional hypothesis that $-(-K)$ is located.%

\medskip
%

\bigskip
%

\noindent
\textbf{Acknowledgement. \ \ }Theorem \ref{mar20t2} arose from discussions
with Hajime Ishihara about thirty years ago.%

\bigskip

%

\bigskip
%

\bigskip
%

\noindent
\textbf{Author's address: \ }School of Mathematics and Statistics, University
of Canterbury, Christchurch 8140, New Zealand%

\noindent
\textbf{email}: dsbridges.math@gmail.com

\end{document}